\documentclass[smallextended,natbib,runningheads, referee]{svjour3}

\journalname{Annals of the Institute of Statistical Mathematics}

\smartqed  

\usepackage{amsmath}
\usepackage{amssymb}
\usepackage{amsthm}
\usepackage{amsfonts}
\usepackage{verbatim}
\usepackage{graphicx}

\newcommand{\multinom}[2]{\genfrac{[}{]}{0pt}{}{#1}{#2}}
\newcommand{\bn}{\mathbf{n}}
\newcommand{\br}{\mathbf{r}}
\newcommand{\bp}{\mathbf{p}}
\newcommand{\bq}{\mathbf{q}}
\newcommand{\bm}{\mathbf{m}}

\newcommand{\bone}{\mathbf{1}}
\newcommand{\hmo}{\bar{h}}
\newcommand{\hmt}{\tilde{h}}
\newcommand{\Qmo}{\bar{Q}}

\newcommand{\Hmo}{\bar{H}}
\newcommand{\Q}{Q(\bn; q; m)}

\begin{document}

\title{The $m$-th longest runs of multivariate random sequences}
\author{Yong Kong}


\titlerunning{The $m$-th longest runs}        

\institute{
  Yong Kong \at
  Department of Biostatistics \\
  School of Public Health \\
  Yale University \\
  New Haven, CT 06520, USA \\
email: \texttt{yong.kong@yale.edu} }

\date{}

\maketitle

\newpage

\begin{abstract}

The distributions of the $m$-th longest runs 
of multivariate random sequences are considered.
For random sequences made up of $k$ kinds of letters,
the lengths of the runs are sorted in two ways to give two 
definitions of run length ordering.
In one definition,
the lengths of the runs are sorted separately for each letter type.
In the second definition, the lengths of all the runs are sorted together. 
Exact formulas are developed for the distributions of the $m$-th longest runs
for both definitions.
The derivations are based on a two-step method that is applicable
to various other runs-related distributions,
such as joint distributions of several letter types
and multiple run lengths of a single letter type.


\end{abstract}
\newpage

\section{Introduction} \label{S:intro}

Run statistics have been used in various disciplines
to test the nonrandomness in sequences 
\citep{Balakrishnan2002,Godbole1994,Knuth1997b}.
For the related topic of scan statistics see for example 
\citep{Glaz2001}.
The research in this area has been revived recently
because of the applications in biological related problems,
such as sequence analysis and genetic analysis.

One of the commonly used runs tests is the longest run test.
An unusual long consecutive appearance of one type of letter 
usually indicates the nonrandom nature of the process
that generates the sequence.
These long consecutive appearances (runs), however,
are usually obscured by  noises or other processes
so that they don't reveal themselves to the observers,
making the application of the longest run test difficult or impossible.
For example, when we consider biological sequences
such as DNA sequences, the longest run of one
particular letter type might have been broken into several
shorter ones due to either biological mutations
or errors that occurred in the process of reading out these sequences.
For such sequences, 
it would be difficult to have a single run of statistically significant 
length. 
For example, for a binary system of $17$ total elements,
with $10$ elements of the first letter type and 
$7$  elements of the second letter type,
the longest run of the first letter type
needs a length of $7$ to achieve statistically significance
with a cutoff of $\alpha=0.05$:
$P(l_0 \ge 7) = 0.049$ (see Eq.~\eqref{E:N_m_ge_q_3}).
On the other hand, if we use the second longest run,
it requires only $l_1 \ge 4$ to achieve statistically significance
for the same significance level:
$P(l_1 \ge 4) = 0.041$ (see again Eq.~\eqref{E:N_m_ge_q_3}).
One of the goals of this paper is to develop 
explicit, easily calculated  formulas for the $m$-th longest runs of
multivariate random sequences,
where $m$ is an arbitrary nonnegative integer.
As shown in Figure~\ref{F:W200_300}
and Table~\ref{T:avg_var}, as $m$ becomes bigger,
the distributions become narrower, 
so it might become easier to tell whether the observed statistic
 comes from one distribution or the other.

The distributions of the longest runs and other runs-related distributions
have been studied by previous researchers
for independent trials and Markov dependent trials
\citep{Burr1961,Philippou1985,Philippou1986,Schilling1990,Koutras1993,Koutras1995,Lou1996,
Muselli2000,fu2003b,Eryilmaz2006,Makri2007}.
Based on the results of \citet{Mood1940},
a  distribution is derived which gives probability of
at least one run of a given length or greater for the special
case of binary systems where each kind of object has
the same number of elements ($n_1 = n_2$)
\citep{Mosteller1941}. 
The formulas involve double summations. 
These results are simplified later
\citep[pp. 255-259]{,Olmstead1958,Bradley1968},
again for binary systems. 
A recursion based algorithm is given for the distribution of 
the longest run of any letter type 
of multiple object systems \citep{Schuster1996}.
\citet{Morris1993} 
had similar objectives to ours:
they obtained exact, explicit formulas for 
multiple objects containing any minimum collection of specified 
lengths.
Their formulas were obtained by using convoluted combinatorial
arguments.
Different from their approach, we will use a simple method
that can derive various distribution in a unified, almost mechanical
way. This is the second major goal of this paper:
to introduce a systematic method 
that can treat various distributions in a unified approach.

Earlier studies of run statistics usually used  ingenious \emph{ad hoc} 
combinatorial methods, which sometimes became very tedious. 
In \citep{Kong2006} a systematic method to study various
run statistics in multiple letter systems was developed.
Two of the commonly-used run tests,
the total number of runs test
and the longest run test, were investigated in detail by using the
general method.
The method was later applied to other commonly used runs tests
\citep{Kong2013,Kong2014,Kong2015}.
In this paper, 
we extend the method to investigate the distributions
of the $m$-th longest runs for multi-letter systems.
Two different definitions of the run length order
will be studied.
For the first definition, 
the lengths of the runs are sorted separately for each letter type.
The formula developed in 
\citep[Theorem 10]{Kong2006} is a special case
for this definition, with $m_i=0$ for each letter type.
For the second definition,
the lengths of all the runs from all letter types are sorted together.
The distributions of both definitions can be considered as special cases
of the general two-step method discussed in Section~\ref{S:2-step}.

In the general setting the method involves two steps.
In the first step,
we only need to consider
the arrangements of a \emph{single} letter type 
that meet the restrictions 
we impose on that letter type, such as the lengths of the runs
or the number of runs, without worrying about the
complicated combinations with other letter types.
It simplifies the enumeration tasks considerably
when only one letter type is considered.
The number of such arrangements of a single letter type, say the $i$-th type,
with $n_i$ elements and $r_i$ number of runs, 
is specified in Eq.~\eqref{E:gen2} by $U(n_i, r_i; X_i)$,
where $X_i$ is a place-holder for other restrictions
in addition of $n_i$ and $r_i$.
In the second step, 
the quantities $U(n_i, r_i; X_i)$, which are for individual letter types,
are then combined together by the function $F(\br)$
(as shown in Eq.~\eqref{E:F}) to get the distribution of the whole system.
The function $F(\br)$ gives the number of configurations
to arrange in a line 
$r_1$ blocks of the first letter type,  
$r_2$ blocks of the second letter type,  etc.,
without the blocks of the same letter type touching each other.
The explicit expression of $F(\br)$ (Eq.~\eqref{E:F})
makes it possible to obtain explicit expressions for various kind of
run-related distributions.
These expressions can often be simplified
by manipulating binomial and multinomial coefficients,
which can be done mechanically using the 
Wilf-Zeilberger method \citep{AEQB}.

The major results of this article are Theorem~\ref{Th:N1}
and Theorem~\ref{Th:N2}.
These theorems are for the $m$-th longest runs in systems 
with arbitrary number of letter types
under the two different definitions of the run length ordering.
Both results are obtained by using the simple yet quite general
consideration discussed in Theorem~\ref{Th:gen2}.
It's interesting to note that, when compared to
the formulas for the special case of the longest runs ($m=0$),
the only difference is that 
the general formulas 
for the arbitrary $m$-th longest runs contain an extra binomial factor
$(-1)^m \binom{j-1}{m}$ where $j$ is a summation variable
(see Eqs.~\eqref{E:hm}, \eqref{E:Hm}, 
\eqref{E:N_m_ge_q_3}, and \eqref{E:Qm_mix}).

There are many definitions of runs in the literature. 
In this article
we use the classical definition of \citet{Mood1940}, 
which asserts
that consecutive runs of one letter type must be separated by other 
letter types.
This is also the definition we used in the previous work \citep{Kong2006}.

Two kinds of models are usually used
when the distributions of runs are studied.
If the numbers of elements for each letter type are fixed,
the models are known as \emph{conditional} models.
If the elements are not fixed but chosen from a multinomial population,
the models are called \emph{unconditional}.
For both of these  models, 
exact finite distributions and asymptotic distributions have
been investigated in the past.  
The results presented in this article are 
exact distributions conditioned on the compositions of systems under study,
i.e., the numbers of each letter type are fixed.
These exact distributions are particular useful for relatively short sequences
and other situations where asymptotic results cannot be applied.
Once the conditional distribution is obtained,
it is usually easy to get unconditional distribution 
by the multinomial theorem.

Throughout the article we reserve the letter $k$ for the number of 
letter types in the system, 
and use $n_i$ as the number of elements of the $i$-th letter type.
The total number of elements of the system is $n = \sum_{i=1}^k n_i$.
The letter $m$ (with index if necessary) 
is used to indicate the run order.
For the first definition of the run ordering, 
$m_i=0$ is used to index the longest run of the $i$-th letter type, 
and $m_i=1$ is the index of the second longest run, etc.
For the second definition, since the letters are pooled together
when the run lengths are ordered, the subscript on $m$ is no longer needed.
In this case, $m=0$ indicates the longest run of the whole system,
and $m=1$ is the second longest run, etc.

We denote by the bold letters 
the tuples with $k$ elements, such as
$\bn = (n_1, n_2, \dots, n_k)$,
$\br = (r_1, r_2, \dots, r_k)$,
$\bp = (p_1, p_2, \dots, p_k)$, and similarly for other symbols.
We use $\binom{n}{m}$ for the binomial coefficient ($n$ choose $m$),
and $\multinom{p_1+\cdots+p_k}{p_1, \dots, p_k} = 
\multinom{p}{p_1, \dots, p_k} = \multinom{p}{\bp}
= p! /(p_1! \cdots p_k!)$ as the multinomial coefficient,
with $p = \sum_{i=1}^k p_i$.
When there is no ambiguity, the $k$ nesting summations
will be abbreviated as a single sum for clarity, for example
$\sum_{p_1 = 1}^{r_1} \cdots \sum_{p_k = 1}^{r_k} f(\bp)$ 
will be written as $\sum_{p_i = 1}^{r_i} f(\bp)$.
The coefficient of $x^m$ of a polynomial $f(x)$ is denoted as
$[x^m] f(x)$.

The paper is organized as follows.
In Section~\ref{S:2-step} we  describe the two-step method 
outlined above in a general setting.
Then in Section~\ref{S:mLong} and Section~\ref{S:2nd-def}
the method is applied to obtain the distributions of
the $m$-th longest runs,
under two different definitions of the run length ordering.

\section{A general two-step method for run-related distributions}
\label{S:2-step}

As discussed in Section~\ref{S:intro},
in the first step we only consider the enumeration of
one letter type when its elements are considered alone.
The enumeration of one letter type is considerably easier than
when all the letter types are considered together.
Let assume that for the $i$-th letter type with $n_i$ elements arranged
into $r_i$ runs, we impose one or more additional conditions,
collectively denoted as $X_i$. 
Denote  $U(n_i, r_i; X_i)$ as the number of arrangements
of $n_i$ elements of the $i$-th letter type 
in exactly $r_i$ runs with the additional restriction $X_i$ imposed.
Various methods can be used to obtain $U(n_i, r_i; X_i)$,
with generating function as one of the most powerful
and versatile methods 
(see Eq.~\eqref{E:GF_hm1} for one of such applications).

After obtaining $U(n_i, r_i; X_i)$, we need to put them together
to form a $k$-letter type system.
To do this we will use
the function $F(\br)$,
which is the number of ways
to arrange in a line $r_1$ runs of the first letter type, 
$r_2$ runs of the second letter type, etc., 
without two adjacent runs being of the same kind.
The explicit expression of function $F(\br)$ 
is given by \citep{Kong2006}:
\begin{lemma}
The function $F(\br)$ is given by
\begin{equation}\label{E:F}
  F(\br) = 
  \sum_{\substack{1 \leq p_i \leq r_i\\1 \leq i \leq k}}
  (-1)^{\sum_i(r_i - p_i)} 
          \binom{r_1 - 1}{p_1 - 1}
          \dots
          \binom{r_k - 1}{p_k - 1}
          \multinom{p_1 + \dots + p_k}{p_1, \dots, p_k}.               
\end{equation}
\end{lemma}
When $k=2$, $F(r_1, r_2)$ can be simplified from Eq.~\eqref{E:F}
to the following trivial expression,
\begin{equation*} 
  F(r_1, r_2) = \binom{2}{r_1 - r_2 + 1} = 
  \begin{cases} 
    2 & \text{if $r_1 = r_2$,} \\
    1 & \text{if $ | r_1 - r_2 | = 1 $,} \\
    0 & \text{otherwise,}
  \end{cases}
\end{equation*}
which is obvious from the meaning of function $F(\br)$.

\begin{proposition}
For a system with $k$ letter types, 
the total number of configurations is given by
\begin{equation}\label{E:g1}
 R(\bn) = \sum_{r_i=1}^{n_i} F(\br) 
                     \prod_{i=1}^k U(n_i, r_i; X_i) , 
\end{equation}
where
$U(n_i, r_i; X_i)$ is the number of arrangements
of $n_i$ elements of the $i$-th letter type 
in exactly $r_i$ runs with restrictions $X_i$ imposed.
\end{proposition}

Often the time we don't want to impose the restrictions
on all of the $k$ letter types.
For example we might only be interested in the length of runs of
the first letter type, and put no restrictions on the other $k-1$ letter types.
Or we are only interested in the length of runs of
the first and the second letter types
to obtain their joint distributions.
In general, suppose 
we only impose certain restrictions on some of the
$k$ letter types, which are indexed
by $S = \{i_1, i_2, \dots,  \}$.  
Then the number of configurations of the system
 $R(\bn; S)$ can be written as
\begin{equation}\label{E:general}
 R(\bn; S) = \sum_{r_i} F(\br) 
                     \prod_{i \in S} U(n_i, r_i; X_i) 
                     \prod_{i \notin S} V(n_i, r_i) ,
\end{equation}
where $V(n_i, r_i)$ is the number of arrangements
of the $n_i$ elements of the $i$-th letter type 
in exactly $r_i$ runs without any restrictions.

%
%
\begin{theorem} 
  \label{Th:gen2}
The number of configurations $R(\bn; S)$ for a system with $k$ 
letter types
 is given by
\begin{equation}\label{E:gen2}
 R(\bn; S) = \sum_{r_i,p_i, i \in S} 
 \multinom{\sum_{i \in S} p_i + 
   \sum_{i \notin S} n_i}{p_i, n_i}
 \prod_{i \in S} 
 (-1)^{r_i - p_i} 
 \binom{r_i - 1}{p_i - 1}
 U(n_i, r_i; X_i) 
,
\end{equation}
where the set $S$ specifies the subset of letter types
on which additional restrictions are imposed.
\end{theorem} 
\begin{proof}
The expression of  $V(n_i, r_i)$ in Eq.~\eqref{E:general} for the unrestricted arrangements of
exactly $r_i$ runs using $n_i$ elements is given by the well-known formula
\begin{equation}\label{E:V}
 V(n_i, r_i) = \binom{n_i-1}{r_i-1} .
\end{equation}
A direct interpretation of above expression is
to put $r_i-1$ bars between the $n_i-1$ spaces formed by the $n_i$ elements
to form $r_i$ runs. 
By using Eqs.~\eqref{E:F} and \eqref{E:V} and utilizing the identity
\[
  \sum_{r=0}^{n} (-1)^r \binom{n}{r} \binom{r}{m} = (-1)^n \delta_{n,m},
\]
the sums of $r_i$ in Eq. \eqref{E:general} for $i \notin S$ can be
evaluated to $(-1)^{n_i} \delta_{n_i, p_i}$.
These $\delta_{n_i, p_i}$ in turn filter out the sums of $p_i$ 
in the explicit expression of $F(\br)$  
for $i \notin S$ to a single term with $p_i = n_i$, leading to
the simplification of $R(\bn; S)$ to sums that only involve
letter types in $S$.
\end{proof}

With different assignments of the set $S$, Theorem~\ref{Th:gen2} can be used 
to obtain different kinds of distributions, such as
joint distributions of two or three letter types,
with $S=\{1,2\}$ and $S=\{1,2, 3\}$ respectively.
Several special cases for this theorem are mentioned here for:
(1) $|S|=0$, 
(2) $|S|=k$, and
(3) $|S|=1$. 
If $S$ is empty, then $R(\bn; S)$ is simplified to the trivial result
$\multinom{n}{n_i}$, as it should be:
\[
 R(\bn; S= \varnothing) = \multinom{n}{n_i} .
\]

If $S = \{1,2, \dots, k\}$, i.e., all letter types are subject to restrictions, 
then
\begin{equation} \label{E:all}
 R(\bn; S = \{1, \dots, k\}) = 
 \sum_{p_i}
 (-1)^{p_i} 	
 \multinom{p}{p_i}
 \sum_{r_i} 
\prod_{i}  
 (-1)^{r_i} 
  \binom{r_i - 1}{p_i - 1}
  U(n_i, r_i; X_i) 
.
\end{equation}
If only one letter type has restrictions, say $S=\{1\}$, then
\begin{align*}
 R(\bn; S=\{1\}) &=
\sum_{r_1=1}^{n_1} \sum_{p_1=1}^{r_1} (-1)^{r_1-p_1}
 \multinom{n-n_1+p_1}{p_1, n_2, \cdots, n_k} \binom{r_1 - 1}{p_1 - 1} U(n_1, r_1; X_1) \\
 &= \multinom{n-n_1}{n_2, \cdots, n_k}
  \sum_{r_1=1}^{n_1} \sum_{p_1=1}^{r_1} (-1)^{r_1-p_1}
  \binom{n-n_1+p_1}{p_1} \binom{r_1 - 1}{p_1 - 1} U(n_1, r_1; X_1) .
\end{align*}
By using the identity
\[
 \sum_{p=1}^r (-1)^p \binom{n+p}{p} \binom{r-1}{p-1} = (-1)^r \binom{n+1}{r},
\]
we get
\begin{corollary} 
  \label{C:k2}
For a system with the first letter type restricted
while the other letter types are unrestricted, the number of configurations
is given by
\begin{equation}\label{E:gen-k2}
 R(\bn; S=\{1\}) = \multinom{n-n_1}{n_2, \cdots, n_k}
\sum_{r_1=1}^{n_1} \binom{n-n_1 + 1}{r_1} U(n_1, r_1; X_1) .
\end{equation}
\end{corollary} 
The direct 
interpretation of Eq. \eqref{E:gen-k2}
is that the $n-n_1$ elements of the other letter types form $n -n_1 + 1$ 
intervals in a line (including the two ends).
There are $\binom{n-n_1 + 1}{r_1}$ ways for the elements of
the first letter type to choose
$r_1$ out of these $n-n_1 + 1$ intervals to form
$r_1$ runs. The multinomial factor in the front takes care of
the number of configurations the elements of the other letter types
can form among themselves.

In the following we will use this two-step method
to derive distributions of the $m$-th longest runs
under two different definitions.

\section{The first definition of the $m$-th longest run: 
run lengths sorted within each letter type} \label{S:mLong}
In this definition, the run lengths are sorted for each letter type
separately.  
For the $i$-th letter, we denote 
$l_{0}^{(i)}$ as the length of the longest run of the $i$-th letter type,
$l_{1}^{(i)}$ as the length of the second longest run of the $i$-th letter 
type, etc.
In general, $l_{m}^{(i)}$ is the length of 
the $(m+1)$-th longest run of the $i$-th letter type.
The lengths of all the runs formed by the $i$-th letter type are ordered as
\[
l_{0}^{(i)} \geq l_{1}^{(i)} \geq l_{2}^{(i)} \geq \cdots 
\geq l_{r_i - 1}^{(i)} .
\]
In other word, there are 
at least
$m+1$ runs of the $i$-th letter type whose length is longer or equal to 
$l_{m}^{(i)}$.

For example, in a $k=4$ system made up of letter types 
$\{1,2,3,4 \}$, if we have the following particular arrangement of the four
letter types 
\begin{equation} \label{E:example}
   111 \,|\, 2 \,|\, 111 \,|\, 333 \,|\, 444444 \,|\, 33 \,|\, 11111 ,
\end{equation}
then we have 
$l_{0}^{(1)} = 5$, 
$l_{1}^{(1)} = 3$, 
$l_{2}^{(1)} = 3$ for the first letter type $1$'s,
$l_{0}^{(2)} = 1$ for the second letter type $2$'s,
$l_{0}^{(3)} = 3$, 
$l_{1}^{(3)} = 2$ for the third letter type $3$'s,
and $l_{0}^{(4)} = 6$ for the fourth letter type $4$'s.
All other $l_{m}^{(i)} = 0$.

As described in Section~\ref{S:2-step},
to use the two-step method to obtain the distribution of the whole system
we first focus on a single particular letter type. 
In the following if we only deal with one letter type,
the index $i$ in $l_{m}^{(i)}$ will be omitted and
we will use $l_{m}$ for simplicity. 
Define function $h_m(n,q,r)$ as the number of ways to arrange 
the elements of a given letter type with $n$ elements in $r$ runs, 
with the length of $(m+1)$-th
longest run less than or equal to $q$, $q \ge 0$, i.e., $l_m \le q$.
In other word, at most $m$ runs can have lengths greater than $q$.
This is a specialization of 
the generic function $U(n,r; X)$ of Eq.~\eqref{E:general},
with the parameters $m$ and $q$ jointly act as the restriction
parameter $X$.
In the following we will find an explicit expression for $h_m(n,q,r)$.

By definition it is obvious that
\begin{equation} \label{E:rec}
 h_m(n,q,r) =  \sum_{i=0}^m \hmo_i(n,q,r),
\end{equation}
where 
$\hmo_i(n,q,r)$ counts for the arrangements which
have exactly $i$ runs whose lengths are greater than $q$.
We'll first find an explicit  expression for
$\hmo_i(n,q,r)$, then use the above relation to
obtain $h_m(n,q,r)$.

\begin{lemma}
  \label{L:h1}
The number of ways to arrange $n$ elements in $r$ runs with
exactly $m$ longest runs of length 
greater than $q \ge 0$ is given by
\begin{equation} \label{E:h1}
 \hmo_m(n,q,r) = 
 \begin{cases}
  0                 & \text{$q = 0$ and $r \ne m$}, \\ 
  \binom{n-1}{r-1}   & \text{$q = 0$ and $r = m$}, \\
  \sum\limits_{j=m}^{\min(r, \lfloor (n-r)/q \rfloor)} 
  (-1)^{m+j} \binom{j}{m} \binom{r}{j}\binom{n - q j -1}{r-1}
 & \text{otherwise.}
 \end{cases}
\end{equation}
\end{lemma}
\begin{proof}
To calculate $\hmo_m(n,q,r)$, we define generating function
$g(x, y, q)$ as
\begin{equation} \label{E:GF_hm1}
  g(x, y, q) = (x + \cdots + x^q) + y (x^{q+1} + \cdots) 
= \frac{x(1-x^q)}{1-x} + y \frac{x^{q+1}}{1-x}.
\end{equation}
If we expand $g(x, y, q)^r$, then $\hmo_m(n,q,r)$ will be
the coefficient of $x^n y^m$,  
since this term counts the number of configurations with
exactly $m$ runs whose lengths are greater than $q$ for
a total of $n$ elements. We obtain:
\begin{align} \label{E:hm1}
 \hmo_m(n,q,r) &= [x^n y^m] g(x, y, q)^r \notag \\
 &= [x^n y^m] (1-x)^{-r} \left[ x(1-x^q) + y x ^{q+1} \right]^r \notag \\
 &= [x^n] (1-x)^{-r} \binom{r}{m} 
         \left[ x(1-x^q) \right]^{r-m} x^{m(q+1)} \notag \\
 &= [x^n]  \binom{r}{m} \sum_{l} \binom{r+l-1}{r-1} 
            \sum_j (-1)^j \binom{r-m}{j} x^{l+r-m+qj+m(q+1)} \notag \\
 &= \binom{r}{m} \sum_{j} (-1)^j \binom{r-m}{j} \binom{n - (m+j)q-1}{r-1} 
	    \notag \\
 &= \sum_{j} (-1)^{j-m} \binom{j}{m} \binom{r}{j}\binom{n - q j -1}{r-1} .
\end{align}
Eq.~\eqref{E:hm1} includes the special case of $q=0$, which
can be checked explicitly.
For $q=0$, we need to put 
$n$ elements into $r$ runs with exactly $m$ runs whose lengths are
greater than zero.
Each run, by definition,
has a length greater than zero.
Hence for $q=0$, $\hmo_m(n,q,r)$ vanishes for all values of
$m$ except for  $m=r$.
This is also reflected in Eq.~\eqref{E:GF_hm1}:
when $q=0$, the only term of $y$ in $g(x, y, q)^r$ is $y^r$.
In this case there are $\binom{n-1}{r-1}$ number of ways
to arrange $n$ elements into $r$ runs.
This can be checked in Eq.~\eqref{E:hm1}:
the sum has only one nonvanishing term, which is when $j=r=m$,
leading to $\binom{n-1}{r-1}$.
\end{proof}

\begin{lemma}
  \label{L:hm}
The number of ways to arrange $n$ elements 
in $r$ runs, with the length of $(m+1)$-th
longest run less than or equal to $q \ge 0$,
 is given by
\begin{equation}\label{E:hm}
 h_m(n,q,r) = 
 \begin{cases}
   0                & \text{$q = 0$  and $r > m$,} \\
   \binom{n-1}{r-1} & \text{$q = 0$  and $r \le m$,}\\
   \sum\limits_{j=0}^{\min(r, \lfloor (n-r)/q \rfloor)}  
   (-1)^{m+j} \binom{j-1}{m} \binom{r}{j} 
   \binom{n-qj-1}{r-1} & \text{otherwise.}
 \end{cases}
\end{equation}
\end{lemma}
\begin{proof}
From the  relation Eq.~\eqref{E:rec} and expression of
$\hmo_m(n,q,r)$ in Eq.~\eqref{E:h1}, we have
\begin{align} \label{E:hmnqr}
 h_m(n,q,r) &=  \sum_{l=0}^m \hmo_l(n,q,r) = \sum_{j} (-1)^{j} 
 \binom{r}{j}\binom{n - q j -1}{r-1} \sum_{l=0}^m (-1)^l \binom{j}{l} 
 \notag\\
 &= (-1)^m \sum_{j} (-1)^{j} \binom{r}{j}\binom{n - q j -1}{r-1} \binom{j-1}{m} .
\end{align}
By definition,
$h_m(n,q,r) = \binom{n-1}{r-1}$ when $m \ge r$,
the number of configurations with $n$ elements in $r$ runs.
This is reflected in the above expression as
$j=0$ is the only nonvanishing term in the sum when $m \ge r$.
As before, 
some special cases when $q=0$ should be considered. Apparently when $q=0$,
$h_m(n,q,r) = 0$ if $r>m$. When $q=0$ and $r \le m$, Eq.~\eqref{E:hmnqr}
is simplified to $\binom{n-1}{r-1}$.
\end{proof}

With the expression of $h_m(n,q,r)$ in Lemma \ref{L:hm},
we can use Eq.~\eqref{E:gen2} in Theorem~\ref{Th:gen2}
to get the distribution for the whole system. 
First, we define two sets of numbers $\bm = (m_1, \dots, m_k)$
and $\bq = (q_1, \dots, q_k)$.
Then
we denote $N(\bn; \bq; \bm)$
as the number of ways to have 
the $(m_i + 1)$-th longest run of the $i$-th letter type
\emph{equal to or less than} $q_i \ge 0$ 
for \emph{all} letters: $i = 1, \dots, k$.
In other word,
$N(\bn; \bq; \bm)$ is the number of ways 
to arrange the letters so that $\forall i  \in \{1,2, \dots, k\}$, 
$l_{m_i}^{(i)} \leq q_i$.
The Theorem 10 of \citep{Kong2006} is a special case of
$N(\bn; \bq; \bm)$ with $\bm = (0, \dots, 0)$, i.e., 
only the longest run for each letter type is considered there.
From  Eq.~\eqref{E:general} we have
\begin{equation}\label{E:Nmhm}
 N(\bn; \bq; \bm) =  \sum_{r_i = 1}^{n_i} F(\br) \prod_i h_{m_i}(n_i,q_i,r_i) .
\end{equation}
Eq.~\eqref{E:Nmhm} can be simplified if we use the explicit expression of 
$F(\br)$, as in Eq.~\eqref{E:all}.
If we put $U(n,r,X)=h_m(n,q,r)$ in Eq.~\eqref{E:all} and
define the last sum in Eq.~\eqref{E:all} as
\[
 H_m(n,q,p) = \sum_r (-1)^r \binom{r-1}{p-1} h_m(n,q,r),
\]
then the summation
of the running variable $r$ 
can be carried out and we have
\begin{theorem} 
  \label{Th:N1}
The number of ways 
to arrange the $k$ letter types so that for 
$\forall i  \in \{1,2, \dots, k\}$, 
$l_{m_i}^{(i)} \leq q_i$ is given by
\begin{equation}\label{E:NmHm}
 N(\bn; \bq; \bm) =  \sum_{p_i = 1}^{n_i} (-1)^{p_i} 
 \multinom{\sum p_i}{p_i} \prod_i H_{m_i}(n_i,q_i,p_i) ,
\end{equation}
where
\begin{equation}\label{E:Hm}
 H_m(n,q,p)= 
 \begin{cases}
  (-1)^m \binom{n-1}{p-1} \binom{n-p-1}{m-p} & 
   \text{$q = 0$ and $n \le m$,}\\
  \binom{n-1}{p-1} \left[ (-1)^m \binom{n-p-1}{m-p} - \binom{n-p-1}{n-1} 
    \right] &
   \text{$q = 0$ and $n > m$,}\\
   \sum\limits_{j=\lceil (n-p)/(q+1) \rceil}
                 ^{\lfloor (n-p)/q \rfloor} 
   (-1)^{n+m+qj+j} \binom{j-1}{m} 
   \binom{n-qj-1}{p-1} \binom{p}{n-qj-j} & \text{otherwise.}
 \end{cases}
\end{equation}
\end{theorem}
The special case of $\bm=(0, \dots, 0)$ has been reported previously
\citep[Theorem 10]{Kong2006}.  
Comparing the two expressions we see that 
the only difference is the extra binomial term
$(-1)^m \binom{j-1}{m}$ for the general case of the $m$-th longest runs
 in Eq.~\eqref{E:Hm}.

If we define $L(\bn; \bq; \bm)$ as 
the number of arrangements to have at least one of the $k$ letter types,
for example, the $i$-th letter type,
to have the length of the $(m_i+1)$-th longest run equal to $q_i$, i.e.,
$\exists i \in \{1,2,\dots,k\}$, $l_{m_i}^{(i)} = q_i$, then
by definition,
$L(\bn; \bq; \bm) = N(\bn; \bq; \bm) - N(\bn; \bq - \bone; \bm)$.
\begin{corollary}
 The number of arrangements to have at least one of the $k$ letter types
 to have the length of the $(m_i+1)$-th longest run equal to $q_i$ is given by
\[
 L(\bn; \bq; \bm) = N(\bn; \bq; \bm) - N(\bn; \bq - \bone; \bm).
\]
\end{corollary}

If we define $W(\bn; \bq; \bm)$ as
the number of arrangements for all letter types to have  
the length of the $(m_i+1)$-th longest run equal to $q_i$, then
we have
\begin{corollary}
The number of arrangements for \emph{all} letter types to have  
the length of the $(m_i+1)$-th longest run equal to $q_i$ is given by
\[
 W(\bn; \bm; \bq) = \sum_{p_i = 1}^{n_i} (-1)^{p} 
 \multinom{p}{p_i} \prod_{i=1}^k 
  \left[ H_{m_i}(n_i,q_i,p_i) - H_{m_i}(n_i,q_i-1,p_i) \right] ,
\]
where $p=\sum_{i=1}^k p_i$.
\end{corollary}

Applying Eq.~\eqref{E:hm} in Lemma~\ref{L:hm} to
Eq.~\eqref{E:gen-k2} of Corollary~\ref{C:k2},
we can get the number of configurations of at least $m+1$ runs 
of the first letter type of length $q$ or greater, regardless of
the other letter types:
\[
 Z(\bn; q; m) = \multinom{n}{n_i}
 - \multinom{n-n_1}{n_2, \cdots, n_k}
\sum_{r_1=1}^{n_1} \binom{n-n_1 + 1}{r_1} h_m(n_1, q-1, r_1).
\] 
The summation of $r_1$ in the above equation
can be carried out, leading to
%

\begin{corollary} \label{C:Z}
The number of configurations of at least $m+1$ runs 
of the first letter type with length $q$ or greater is given by
\begin{equation} \label{E:N_m_ge_q_3}
 Z(\bn; q; m) = 
 \multinom{n-n_1}{n_2, \cdots, n_k}
 \sum_{j=1}^{\min(n-n_1+1, \lfloor n/q\rfloor)}
 (-1)^{m+j+1}
 \binom{j-1}{m}
 \binom{n-n_1+1}{j}
 \binom{n-qj}{n-n_1} .
\end{equation}
\end{corollary}
Eq.~\eqref{E:N_m_ge_q_3} is a generalization of previous results,
such as those of 
\citep[p.257]{Bradley1968}.
For the $m$-th longest run, again the only difference
is the extra binomial term $(-1)^m \binom{j-1}{m}$.

By using Theorem~\ref{Th:gen2} and Lemma~\ref{L:hm},
the method can easily lead to joint distributions of various kinds.  
For example, 
instead of using $S=\{1\}$ to focus only on the lengths of runs
of the first letter type, we can use
$S=\{1,2 \}$ to obtain joint distributions
of both the first and the second letter types.
Other possibilities are to introduce more tracking variables
in generating function Eq.~\eqref{E:GF_hm1} to track
more run lengths within one letter type, instead of only one number $q$. 
The details are omitted here.

As for  computational complexity,
Eq.~\eqref{E:Nmhm} has $3$ nested summations over $n_i, i=1, \dots, k$: 
the inner $k$ summations for
$h_m(n,q,r)$ in Eq.~\eqref{E:hm}, the middle $k$ summations for the calculation
of $F(\br)$, and the outer $k$ summations for variables $r_i$.
Hence the computational complexity for Eq.~\eqref{E:Nmhm} 
is $O( \prod_{i=1}^k n_i^{3})$.
Eq.~\eqref{E:NmHm} of Theorem~\ref{Th:N1} simplifies the computation 
to two nested summations over $n_i$, 
and the computational
complexity is reduced to $O(\prod_{i=1}^k n_i^{2})$.

\section{The second definition the $m$-th longest run: run lengths
sorted for all letter types} \label{S:2nd-def}
In Section~\ref{S:mLong}, the $m$-th longest runs are ordered within
runs formed by individual letter types.  In this section, distributions
of $m$-th longest runs of the whole system will be developed.

For this definition the lengths of runs are sorted regardless
which letter type the run is made up of. 
The lengths of runs of \emph{the whole system} are ordered as
$l_0 \ge l_1 \ge \dots \ge l_{r-1}$, where $r$ is the total number of
runs of the system. 
The length of the longest run of the whole system
is $l_0$, 
with the length of the shortest run labeled as $l_{r-1}$.
In general $l_m$ denotes the length of the 
$(m+1)$-th longest run of the whole system.  
We define $l_i = 0$ if $i \ge r$.
If we use the same example shown previously in \eqref{E:example}, then
$l_0 = 6$, 
$l_1 = 5$, 
$l_2 = l_3 = l_4 = 3$, 
$l_5 = 2$, 
$l_6 = 1$, 
and $l_m = 0$ for $m > 6$.

We define $\Q$ as the number of ways to arrange the whole system
to have the length of the $(m+1)$-th longest run less or equal to $q$,
i.e., $l_m \le q$.  The definition of $\Q$ implies that for
all the arrangements counted by $\Q$, there are at most $m$ runs
with lengths greater than $q$.

As before, $\Q$ can be 
expressed by 
\[
 \Q = \sum_{s=0}^m \Qmo(\bn; q; s) ,
\]
where $\Qmo(\bn; q; m)$  is the number of arrangements
of the whole system
where there are exactly $m$ runs with lengths greater than $q$,
regardless of the letter types.
The numbers given by $\Q$ and $\Qmo(\bn; q; m)$
are the corresponding quantities on the whole system level
of the numbers given by $h_m(n, q, r)$ and $\hmo_m(n, q, r)$
discussed in Section~\ref{S:mLong}
for a particular given letter type.

To calculate $\Qmo(\bn; q; m)$, we use the same expression of
$\hmo_m(n, q, r)$ in 
Eq.~\eqref{E:hm1}, which is the number of ways to arrange
$n$ elements of one particular letter type in $r$ runs, with exact
$m$ runs longer than $q$.
Again the function $F(\br)$ is used
to put the whole system together:
\[
 \Qmo(\bn; q; s) = 
 \sum_{\substack{m_i=0\\\sum m_i = s}}^s
  \sum_{r_i = 1}^{n_i} 
  F(\br) 
  \prod_{i} \hmo_{m_i} (n_i, q, r_i) .
\]
Hence for $\Q$ we have
\begin{equation} \label{E:Qm1}
  \Q = \sum_{s=0}^m 
  \sum_{\substack{m_i=0\\\sum m_i = s}}^s
  \sum_{r_i = 1}^{n_i} 
  F(\br) 
  \prod_{i} \hmo_{m_i} (n_i, q, r_i) .
\end{equation}

Eq.~\eqref{E:Qm1} can be simplified.
First, by using the explicit expression of $F(\br)$ of Eq.~\eqref{E:F},
Eq.~\eqref{E:Qm1} can be simplified as
\begin{equation} \label{E:Qm2}
  \Q = \sum_{s=0}^m 
       \sum_{\substack{m_i=0\\ \sum m_i = s}}^s
	\sum_{p_i = 1}^{n_i} (-1)^{ p_i}
		\multinom{\sum p_i}{p_i} \prod_i \Hmo_{m_i} (n_i, q, p_i) ,
\end{equation}
where
\begin{equation} \label{E:Hm_mix}
 \Hmo_m(n,q,p)= 
 \begin{cases}
  (-1)^m \binom{m-1}{p-1} \binom{n-1}{m-1} & 
   q = 0, \\
	\sum\limits_{j=\lceil (n-p)/(q+1) \rceil}
                 ^{\lfloor (n-p)/q \rfloor} 
   (-1)^{n+m+qj+j} \binom{j}{m} 
   \binom{n-qj-1}{p-1} \binom{p}{n-qj-j} & \text{otherwise .}
 \end{cases}
\end{equation}

The expression of Eq.~\eqref{E:Qm2} can be further simplified by
getting rid of the selection summation on $\sum m_i = s$ in the second sum.
Let's discuss the simplification for $q=0$ and $q>0$ separately.

When $q=0$, if $m \ge n$,
we have $Q(\bn; 0; m) = \multinom{n}{n_i}$.
For $q=0$ and $m < n$,
for each summation of $m_i$ in Eq.~\eqref{E:Qm2}, we can first
ignore the selection restriction $\sum m_i = s$,
and use a variable $t$ to track $m_i$ later.
First look at the sum over one particular $m_i$:
\begin{align*}
\sum_{m_i = p_i}^{n_i} & 
            (-1)^{m_i} \binom{m_i-1}{p_i-1} \binom{n_i-1}{m_i-1} t^{m_i}\\
= & \binom{n_i-1}{p_i-1} \sum_{m_i = p_i}^{n_i} (-1)^{m_i} 
                                  \binom{n_i - p_i}{m_i - p_i} t^{m_i}\\
= & (-1)^{p_i} \binom{n_i-1}{p_i-1} (1-t)^{n_i - p_i} t^{p_i} .
\end{align*}
The nested $k$ sums of $m_i$ will then give
\[
 (-1)^p (1-t)^{n - p} t^{p} \prod_{i=1}^k \binom{n_i-1}{p_i-1} ,
\]
where $n = \sum_k n_i$ and $p = \sum_k p_i$.
The selection restriction $\sum m_i = s$ just takes the coefficient of
$t^s$ from the above expression:
\[
 [t^s] (-1)^p (1-t)^{n - p} t^{p} 
 = (-1)^{s} \binom{n-p}{s-p}.
\]
The outmost sum of $s$ can then be carried out:
\[
 \sum_{s=0}^m (-1)^{s} \binom{n-p}{s-p} 
= (-1)^{m} \binom{n-p-1}{m-p}.
\]
Putting all together, we have for $q=0$ and $m<n$,
\begin{equation} \label{E:Qm_mix_0}
 Q(\bn; 0; m) = (-1)^m \sum_{p_i} (-1)^p 
 \binom{n-p-1}{m-p} \multinom{p}{p_i} \prod_i \binom{n_i-1}{p_i-1} .
\end{equation}
Similarly, for $q > 0$ Eq.~\eqref{E:Qm2} can be simplified by first
doing the sums on each $m_i$, and then filtering out the
term with the selection restriction $\sum_{i=1}^k m_i = s$ 
 by taking the coefficient of the $t^s$ term.
In the end, after putting everything together, we obtain
\begin{theorem} 
  \label{Th:N2}
The number of configurations of a system with $l_m \le q$ 
when all lengths of runs are sorted together regardless of
letter types is given by,
when $q > 0$,
\begin{multline} \label{E:Qm_mix}
 \Q = (-1)^{n+m}
  \sum_{p_i=1}^{n_i} (-1)^p  \multinom{p}{p_i} \\
  \times
  \sum_{j_i = \lceil (n_i-p_i)/(q+1) \rceil}
                 ^{\lfloor (n_i-p_i) / q \rfloor} 
                (-1)^{j (q+1)}
                \binom{j-1}{m} 
		\prod_{i}
                \binom{n_i -q j_i -1}{p_i - 1}
                \binom{p_i}{n_i -q j_i -j_i}
\end{multline}
with $j = \sum_i j_i$, $n = \sum_k n_i$, and $p = \sum_k p_i$.
When $q=0$, if $m \ge n$,
\[
 Q(\bn; 0; m) = \multinom{n}{n_i},
\]
when $q=0$ and $m < n$,
\[
 Q(\bn; 0; m) = (-1)^m \sum_{p_i=1}^{n_i} (-1)^p 
 \binom{n-p-1}{m-p} \multinom{p}{p_i} \prod_i \binom{n_i-1}{p_i-1} .
\]
\end{theorem}

If we compare Eq.~\eqref{E:Qm_mix} with Eq.~\eqref{E:NmHm},
we see that the only difference is in the term $(-1)^{m+(q+1)j}\binom{j-1}{m}$:
in Eq.~\eqref{E:NmHm} the term is calculated separately
for individual letter type as $(-1)^{m_i+(q_i+1)j_i} \binom{j_i-1}{m_i}$,
while in Eq.~\eqref{E:Qm_mix} the term is calculated for the
whole system using the $j = \sum_i j_i$.
From the definitions we see that when $m=0$, if we set all
$q_i$ in Eq.~\eqref{E:NmHm} to $q$, so that $\bq = (q, q, \dots, q)$,
$N(\bn; \bq; \mathbf{0}) = Q(\bn; q; 0)$.
This can be confirmed by comparing Eqs.~\eqref{E:NmHm} and \eqref{E:Hm}
with Eq. \eqref{E:Qm_mix}. 
For $m > 0$, this will no longer be true.

\begin{corollary}
  \label{C:W}
The number of ways to have the length of the $(m+1)$-th longest run
as $q$ for the whole system is given by
\[
 W(\bn; q; m) = \Q - Q(\bn;  q-1; m) .
\]
\end{corollary}

As we can see, Eq.~\eqref{E:Qm_mix_0} is very similar in form to Eq. (28) of
\citep{Kong2006}, which calculates the number of configurations
with the \emph{total number of runs} as $r$:
\[
  T(r; \bn) = (-1)^r \sum_{p_i} (-1)^p 
 \binom{n-p}{r-p} \multinom{p}{p_i} \prod_i \binom{n_i-1}{p_i-1} .
\]
By the definition of $\Q$, $Q(\bn; 0; m)$ means 
the number of arrangements to have \emph{at most} $m$ runs with lengths
greater than $0$, i.e., with \emph{at most} $m$ runs.
The relation between $Q(\bn; 0; m)$ and $T(r; \bn)$ is obvious: 
\[
  Q(\bn; 0; m) =  \sum_{r=0}^m T(r; \bn) ,
\]
which can be checked explicitly.

In Figure~\ref{F:W200_300} the probability mass distribution
of  $W(\bn; q; m)$ for $\bn=(n_1, n_2) = (200, 300)$ (divided by
$\binom{n_1+n_2}{n_1}$) is plotted for $m=0$ to $3$.
In Table~\ref{T:avg_var}, the average, the second moment, and
the variance of the same system are listed.
The distributions become narrower when $m$ increases.

\begin{figure} 
  \centering
  \includegraphics[angle=270,width=\columnwidth]{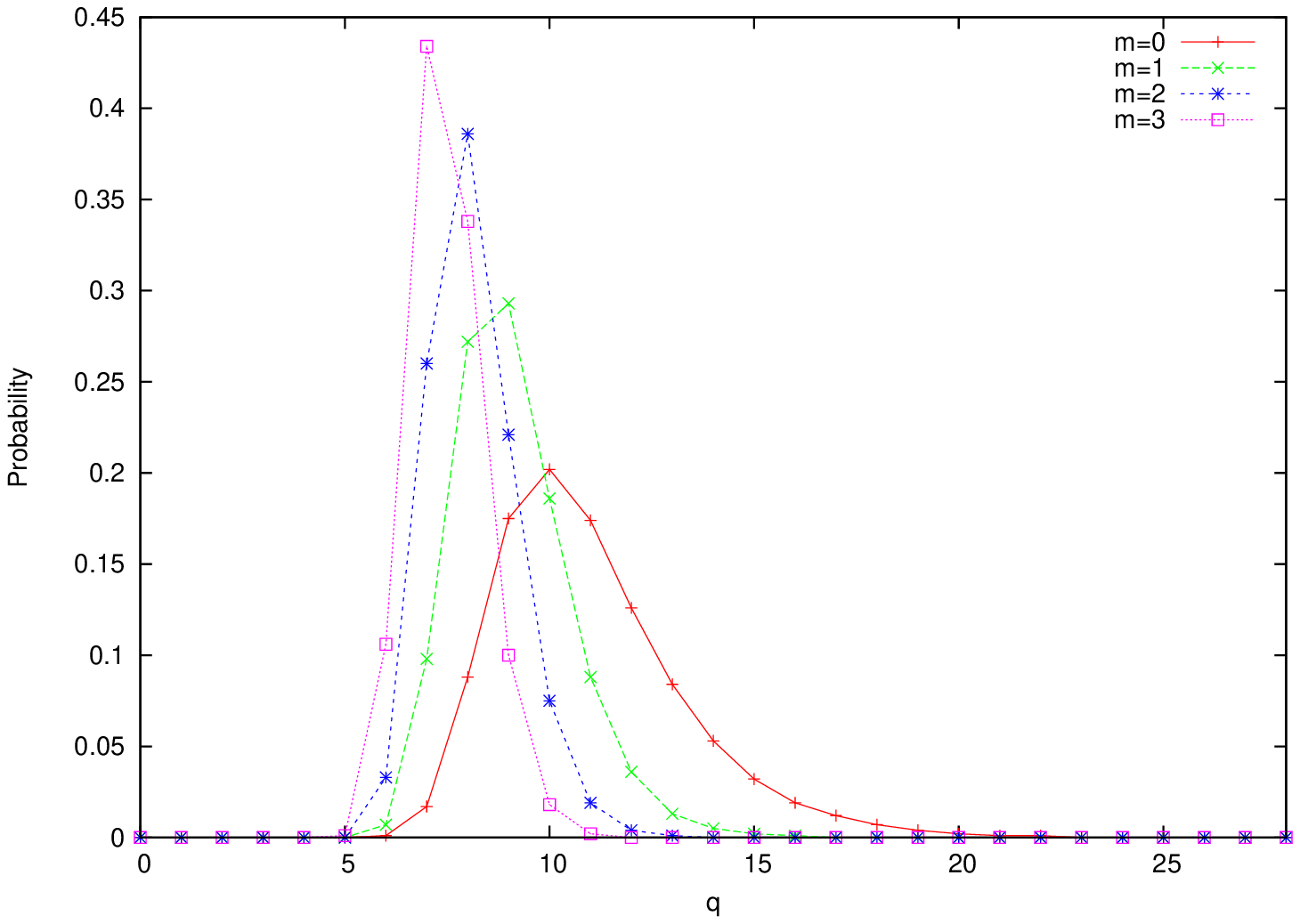}
  \caption{Probability mass distribution of 
the $m$-th longest runs of the whole system, 
for $\bn=(n_1, n_2) = (200, 300)$, $m=0$ to $3$.
The probability is calculated by $W(\bn; q; m)$ in Corollary~\ref{C:W},
 divided by
$\binom{n_1+n_2}{n_1}$.
}
\label{F:W200_300}
\end{figure}

\begin{table} 
\centering
\caption{
Average, the second moment, and
variance of the distribution
in Figure~\ref{F:W200_300}.
}\label{T:avg_var}
\begin{tabular}{ l|r|r|r }
\hline \hline
$m$ & $E(X)$  & $E(X^2)$ & $\sigma^2$ \\
\hline
$0$ & $10.997$  & $126.502$ & $5.562$ \\
$1$ &  $9.072$ &   $84.309$ & $2.006$ \\
$2$ &  $8.121$ &   $67.115$ & $1.165$ \\
$3$ &  $7.494$ &   $56.966$ & $0.809$ \\
\hline
\end{tabular}
\end{table}

\section*{Acknowledgment}
This work was supported in part by 
the Clinical and Translational Science Award UL1 RR024139 
from the National Center for Research Resources, 
National Institutes of Health.

%
%
%



\end{document}